\documentclass[pdflatex,sn-mathphys-num]{sn-jnl}

\usepackage{graphicx}%
\usepackage{multirow}%
\usepackage{amsmath,amssymb,amsfonts}%
\usepackage{amsthm}%
\usepackage{mathrsfs}%
\usepackage[title]{appendix}%
\usepackage{textcomp}%
\usepackage{manyfoot}%
\usepackage{booktabs}%
\usepackage{mathtools}%

\theoremstyle{thmstyleone}
\newtheorem{theorem}{Theorem}
\newtheorem{proposition}[theorem]{Proposition}
\newtheorem{lemma}[theorem]{Lemma}

\theoremstyle{thmstyletwo}
\newtheorem{remark}[theorem]{Remark}
\theoremstyle{thmstylethree}

\DeclareMathOperator{\per}{per}
\DeclareMathOperator{\sgn}{sgn}
\DeclareMathOperator{\diag}{diag}
\DeclareMathOperator{\adj}{adj}
\DeclareMathOperator{\ord}{ord}
\DeclareMathOperator{\Pf}{Pf}

\newcommand{\Fp}{\mathbb F_p}
\newcommand{\Fpt}{\mathbb F_{p^2}}

\newcommand{\Z}{\mathbb Z}

\newcommand{\cR}{\mathcal R}

\newcommand{\Leg}[2]{\left(\frac{#1}{#2}\right)}

\begin{document}

\title[Sun-type determinant and permanent congruences]{Sun-type determinant and permanent congruences}

\author[1]{\fnm{Yaoran} \sur{Yang}}\email{yangyaoran@stu.scu.edu.cn}
\author*[1]{\fnm{Yutong} \sur{Zhang}}\email{yutongzhang@stu.scu.edu.cn}

\affil*[1]{\orgdiv{School of Mathematics}, \orgname{Sichuan University}, \orgaddress{\city{Chengdu}, \postcode{610065}, \country{China}}}

\abstract{Sun proposed a list of congruence and quadratic-residue conjectures for determinants and permanents over residue classes modulo a prime. This article gives a uniform treatment of Conjectures 4.6, 4.7, 4.8(ii), 4.9, 4.10(ii), 4.11 and 4.12 from Sun's list, while making explicit the overlap with two earlier contributions. Luo and Xia's Legendre-symbol formula for $D_p(b,1)$ already implies the non-vanishing assertion in Conjecture 4.6 when $p\equiv5\pmod {24}$; our determinant argument gives a root-quotient criterion for irreducible binary quadratic forms over $\Fp$ and also covers the remaining case $p\equiv19\pmod {24}$. For the Cauchy kernel $1/(x-y)$, we prove the derangement determinant and permanent congruences modulo $p^2$ and a polynomial fixed-point permanent congruence modulo $p$. For the Cayley kernel $(x+y)/(x-y)$, She, Sun and Xia's permanent identity supplies the structural input for the fixed-point permanent; combined with our Cauchy permanent congruence and Morley's congruence, it yields the congruence modulo $p^2$. Independent interpolation arguments give the signed fixed-point determinant congruences and the quadratic-residue assertion for the signed derangement determinant. Finally, a local expansion at the unique zero eigenvalue proves the half-size quadratic Cayley determinant divisibility by $p^2$, and by $p^3$ when $p\equiv7\pmod8$.}



\maketitle

\noindent\textbf{ORCID iDs.} Yaoran Yang: 0009-0004-2832-9163; Yutong Zhang: 0009-0000-1220-0702.\\
\textbf{Corresponding author.} Yutong Zhang.

\section{Introduction}\label{sec:introduction}

Let $A=[a_{ij}]_{1\le i,j\le n}$ be a matrix over a commutative ring. We write
\begin{equation}\label{eq:det-per}
\det A=\sum_{\sigma\in S_n}\sgn(\sigma)\prod_{i=1}^{n}a_{i,\sigma(i)},\qquad
\per A=\sum_{\sigma\in S_n}\prod_{i=1}^{n}a_{i,\sigma(i)}.
\end{equation}
Sun \cite{Sun2024} studied determinant and permanent congruences modulo primes and posed Conjectures 4.1--4.12. The determinant part concerns
\begin{equation}\label{eq:Dab-intro}
D_p(a,b)=\det\left[(i^2+aij+bj^2)^{p-2}\right]_{1\le i,j\le p-1},
\end{equation}
where $p$ is an odd prime and the entries are reduced modulo $p$ when Legendre symbols are taken. The permanent part concerns the kernels
\begin{equation}\label{eq:kernels-intro}
\frac1{x-y},\qquad \frac{x+y}{x-y},\qquad
\frac1{x^2-y^2},\qquad \frac{x^2+y^2}{x^2-y^2}.
\end{equation}
We prove Sun's Conjectures 4.6, 4.7, 4.8(ii), 4.9, 4.10(ii), 4.11 and 4.12. The determinant result is not restricted to the numerical case needed for Conjecture 4.6: it is deduced from a root-quotient criterion for irreducible quadratic forms over $\Fp$. The Cauchy fixed-point permanent is likewise proved in the stronger polynomial form \eqref{eq:perICp-main}.

The determinant assertions in Conjectures 4.2--4.5 are already covered by the Legendre-symbol formulae of Luo and Xia \cite{LuoXia2025}. Their formula for $D_p(b,1)$ also implies the non-vanishing assertion in Conjecture 4.6 when $p\equiv5\pmod {24}$, after the elementary change of variables $j\mapsto sj$ with $s^2\equiv6\pmod p$. We nevertheless include a uniform root-quotient proof of Conjecture 4.6, since the remaining case $p\equiv19\pmod {24}$ is not covered by that argument. For determinant conjectures of Sun and related Legendre-symbol evaluations, see \cite{Chapman2004,Sun2019,WuSheNi2022,LuoSun2023,WangSun2024,SheSun2025,ZhuRen2024,LuoXia2025}.

The cycle and matching reductions used below are also in the line of Cauchy--Borchardt determinant calculus and Pfaffian--Hafnian analogues \cite{Borchardt1855,Krattenthaler1999,IshikawaKawamukoOkada2005}. For permanent identities and root-of-unity versions of Sun-type derangement conjectures, see \cite{Minc1978,GuoLiTaoWei2024,WangSun2023,SheSunXia2022}. The identity of She, Sun and Xia is used explicitly in Proposition \ref{prop:SSX}.

All congruences in this paper are taken in $\Z_p$ or $\Fp$, with denominators invertible in the displayed ranges. Put
\begin{equation}\label{eq:chi-Wilson}
\chi_p=\Leg{-1}{p}=(-1)^{(p-1)/2},\qquad W_p=\frac{(p-1)!+1}{p}.
\end{equation}
Let $D(n)$ denote the set of derangements of $\{1,\ldots,n\}$. Define
\begin{equation}\label{eq:Cp-Rp-def}
C_p=(c_{ij})_{1\le i,j\le p-1},\qquad
c_{ij}=\begin{cases}0,&i=j,\\(i-j)^{-1},&i\ne j,
\end{cases}
\end{equation}
\begin{equation}\label{eq:Rp-def}
R_p=(r_{ij})_{1\le i,j\le p-1},\qquad
r_{ij}=\begin{cases}0,&i=j,\\(i+j)(i-j)^{-1},&i\ne j,
\end{cases}
\end{equation}
and, with $\{1,\ldots,p\}$ read modulo $p$,
\begin{equation}\label{eq:calRp-def}
\cR_p=(\rho_{ij})_{1\le i,j\le p},\qquad
\rho_{ij}=\begin{cases}0,&i=j,\\(i+j)(i-j)^{-1},&i\ne j.
\end{cases}
\end{equation}
If $p\equiv3\pmod4$ and $m=(p-1)/2$, define
\begin{equation}\label{eq:Q-T-def}
Q_p=(q_{ij})_{1\le i,j\le m},\qquad
q_{ij}=\begin{cases}0,&i=j,\\(i^2-j^2)^{-1},&i\ne j,
\end{cases}
\end{equation}
\begin{equation}\label{eq:Tp-def}
T_p=(t_{ij})_{1\le i,j\le m},\qquad
t_{ij}=\begin{cases}0,&i=j,\\(i^2+j^2)(i^2-j^2)^{-1},&i\ne j.
\end{cases}
\end{equation}

\begin{theorem}\label{thm:determinant-main}
Let $p$ be an odd prime. If $p\equiv5\pmod {24}$, then $p\nmid D_p(6,6)$. If $p\equiv19\pmod {24}$, then
\begin{equation}\label{eq:D66-main}
\Leg{D_p(6,6)}{p}\ne -1.
\end{equation}
\end{theorem}

\begin{theorem}\label{thm:main-cauchy}
For every odd prime $p$,
\begin{equation}\label{eq:detCp-main}
\det C_p\equiv1\pmod {p^2},
\end{equation}
\begin{equation}\label{eq:perCp-main}
\per C_p\equiv\chi_p\pmod {p^2},
\end{equation}
and
\begin{equation}\label{eq:perICp-main}
\per(I_{p-1}+uC_p)\equiv1+\chi_pu^{p-1}\pmod p\qquad(u\in\Z_p).
\end{equation}
If $p\equiv3\pmod4$, then
\begin{equation}\label{eq:perIQp-main}
\per(I_m+uQ_p)\equiv1\pmod p\qquad(u\in\Z_p).
\end{equation}
\end{theorem}

\begin{theorem}\label{thm:main-cayley}
For every odd prime $p$,
\begin{equation}\label{eq:perIRp-main}
\per(I_{p-1}+R_p)\equiv ((p-2)!!)^2\pmod {p^2},
\end{equation}
\begin{equation}\label{eq:detIRp-main}
\det(I_{p-1}+R_p)
\equiv \frac{(-1)^{(p+1)/2}}{p-2}((p-2)!!)^2\pmod {p^2},
\end{equation}
and
\begin{equation}\label{eq:detcalRp-main}
\det(I_p+\cR_p)\equiv-\frac p2\pmod {p^2}.
\end{equation}
If $p>3$, then
\begin{equation}\label{eq:R-square-main}
p^{-(3-\chi_p)}\det R_p\in(\Fp)^2.
\end{equation}
\end{theorem}

\begin{theorem}\label{thm:main-half}
Let $p>3$, $p\equiv3\pmod4$, and $m=(p-1)/2$. Then
\begin{equation}\label{eq:detITp-p2}
\det(I_m+T_p)\equiv0\pmod {p^2}.
\end{equation}
If $p\equiv7\pmod8$, then
\begin{equation}\label{eq:detITp-p3}
\det(I_m+T_p)\equiv0\pmod {p^3}.
\end{equation}
\end{theorem}

\begin{remark}\label{rem:scope}
Theorem \ref{thm:determinant-main} is Sun's Conjecture 4.6. The congruences \eqref{eq:detCp-main}, \eqref{eq:perCp-main}, \eqref{eq:perICp-main}, \eqref{eq:perIQp-main}, \eqref{eq:R-square-main}, \eqref{eq:detIRp-main}, \eqref{eq:detcalRp-main}, \eqref{eq:perIRp-main}, \eqref{eq:detITp-p2} and \eqref{eq:detITp-p3} translate into Sun's Conjectures 4.7, 4.8(ii), 4.9, 4.10(ii), 4.11 and 4.12.
\end{remark}

\begin{remark}[Relation with previous work]\label{rem:prior-work}
The first assertion of Conjecture 4.6, namely $p\nmid D_p(6,6)$ for $p\equiv5\pmod {24}$, can also be deduced from Luo and Xia's formula for $D_p(b,1)$ \cite{LuoXia2025}. Indeed, choosing $s\in\Fp^\times$ with $s^2=6$, the column change $j\mapsto sj$ identifies $D_p(6,6)$ modulo $p$, up to sign, with $D_p(s,1)$; since $s^2-4\equiv2\pmod p$ and $p\equiv5\pmod8$, Luo and Xia's criterion applies. Our proof of Theorem \ref{thm:determinant-main} is included to give a uniform root-quotient argument and to cover the assertion for $p\equiv19\pmod {24}$.

The permanent identity used in Proposition \ref{prop:SSX} is Theorem 1.1 of She, Sun and Xia \cite{SheSunXia2022}. Thus the proof of \eqref{eq:perIRp-main} combines that identity with the new congruence \eqref{eq:perCp-main} and Morley's congruence.
\end{remark}

\section{Restricted determinants attached to a binary quadratic form}\label{sec:determinants}

For $a,b\in\Z$, put
\begin{equation}\label{eq:q-def}
q_{a,b}(X)=X^2+aX+b,\qquad \Delta_{a,b}=a^2-4b.
\end{equation}
When the prime $p$ is fixed, all coefficients are reduced modulo $p$. We use quadratic reciprocity in its standard form \cite[Chaps.~5--7]{IrelandRosen1990}.

\begin{proposition}\label{prop:det-root-product}
Let $q_{a,b}$ be irreducible over $\Fp$, and let $\lambda,\mu\in\Fpt$ be its roots. Then
\begin{equation}\label{eq:D-root-product}
D_p(a,b)\equiv\prod_{m=1}^{p-1}(\lambda^m+\mu^m)\pmod p.
\end{equation}
\end{proposition}

\begin{proof}
The roots satisfy $\lambda+\mu=-a$, $\lambda\mu=b$, $\lambda^p=\mu$ and $\mu^p=\lambda$. For $i,j\in\Fp^\times$,
\begin{equation}\label{eq:entry-factor-det}
(i^2+aij+bj^2)^{p-2}=j^{-2}q_{a,b}(i/j)^{-1}.
\end{equation}
Since $\prod_{j\in\Fp^\times}j^{-2}=1$, Wilson's theorem gives
\begin{equation}\label{eq:D-K}
D_p(a,b)\equiv\det K,
\qquad K=\bigl(q_{a,b}(i/j)^{-1}\bigr)_{i,j\in\Fp^\times}.
\end{equation}
For $0\le m\le p-2$, set $v_m=(i^m)_{i\in\Fp^\times}$ and
\begin{equation}\label{eq:Tm-def-det}
T_m=\sum_{x\in\Fp^\times}\frac{x^m}{q_{a,b}(x)}.
\end{equation}
The vectors $v_0,\ldots,v_{p-2}$ form a basis and
\begin{equation}\label{eq:Kv-eigen}
Kv_m=T_{-m}v_m,
\end{equation}
where the subscript is read modulo $p-1$. Hence
\begin{equation}\label{eq:detK-prodT}
\det K=\prod_{m=1}^{p-1}T_m.
\end{equation}
For $\gamma\notin\Fp^\times$ and $1\le m\le p-1$,
\begin{equation}\label{eq:partial-det}
\sum_{x\in\Fp^\times}\frac{x^m}{x-\gamma}=\frac{\gamma^{m-1}}{\gamma^{p-1}-1}.
\end{equation}
Indeed, dividing $X^m$ by $X-\gamma$ and using $\sum_{x\in\Fp^\times}x^r=0$ for $1\le r\le p-2$ and $=-1$ for $r=p-1$, or equivalently differentiating $X^{p-1}-1$, gives \eqref{eq:partial-det}. The partial fraction decomposition
\begin{equation}\label{eq:pf-det}
\frac1{(x-\lambda)(x-\mu)}=\frac1{\lambda-\mu}\left(\frac1{x-\lambda}-\frac1{x-\mu}\right)
\end{equation}
and \eqref{eq:partial-det} give
\begin{equation}\label{eq:Tm-lam-mu}
T_m=-\frac{\lambda^m+\mu^m}{(\lambda-\mu)^2}
=-\frac{\lambda^m+\mu^m}{\Delta_{a,b}}.
\end{equation}
Multiplying \eqref{eq:Tm-lam-mu} for $m=1,\ldots,p-1$ yields \eqref{eq:D-root-product}, since $p-1$ is even and $\Delta_{a,b}^{p-1}=1$.
\end{proof}

\begin{proposition}\label{prop:root-quotient}
Keep the assumptions of Proposition \ref{prop:det-root-product}. Put $R=\lambda/\mu$ and $H=\ord(R)$. If $H$ is even, then
\begin{equation}\label{eq:D-even-order}
D_p(a,b)\equiv0\pmod p.
\end{equation}
If $H$ is odd and $a\not\equiv0\pmod p$, write $p+1=\kappa H$ and $\eta=\lambda^{(p+1)/2}$. Then $\eta\in\Fp^\times$, $\Leg{b}{p}=1$, and
\begin{equation}\label{eq:D-odd-order}
D_p(a,b)\equiv -\frac{2^{\kappa-1}}{a}\eta\pmod p.
\end{equation}
\end{proposition}

\begin{proof}
Since $R^p=R^{-1}$, we have $H\mid p+1$. If $H$ is even, then $R^{H/2}=-1$, and the factor $\lambda^{H/2}+\mu^{H/2}$ in \eqref{eq:D-root-product} is zero.

Assume that $H$ is odd. Then $H\mid(p+1)/2$, so $\lambda^{(p+1)/2}=\mu^{(p+1)/2}$ and this common value is fixed by Frobenius. Thus $\eta\in\Fp^\times$ and
\begin{equation}\label{eq:eta-square}
\eta^2=\lambda^{p+1}=\lambda\mu=b.
\end{equation}
Set
\begin{equation}\label{eq:P-product}
P=\prod_{m=1}^{p-1}(\lambda^m+\mu^m).
\end{equation}
Since $p-1=\kappa H-2$,
\begin{equation}\label{eq:P-rewrite}
P=\mu^{1+2+\cdots+(p-1)}\prod_{m=1}^{\kappa H-2}(1+R^m).
\end{equation}
Because $H$ is odd,
\begin{equation}\label{eq:odd-cyclotomic}
\prod_{r=0}^{H-1}(1+R^r)=2.
\end{equation}
Consequently
\begin{equation}\label{eq:R-product}
\prod_{m=1}^{\kappa H-2}(1+R^m)
=\frac{2^\kappa}{2(1+R^{-1})}
=-\frac{2^{\kappa-1}\lambda}{a}.
\end{equation}
Moreover,
\begin{equation}\label{eq:mu-power}
\mu^{1+2+\cdots+(p-1)}=\mu^{p(p-1)/2}=\lambda^{(p-1)/2}.
\end{equation}
Substituting \eqref{eq:R-product} and \eqref{eq:mu-power} into \eqref{eq:P-rewrite}, and then applying Proposition \ref{prop:det-root-product}, gives \eqref{eq:D-odd-order}.
\end{proof}

\begin{lemma}\label{lem:six-six}
Let $p\equiv5$ or $19\pmod {24}$. Let $\lambda,\mu$ be the roots of $X^2+6X+6$, put $R=\lambda/\mu$, $H=\ord(R)$, and $\eta=\lambda^{(p+1)/2}$. Then $X^2+6X+6$ is irreducible over $\Fp$ and $\Leg{6}{p}=1$. Moreover:
\begin{enumerate}
\item if $p\equiv5\pmod {24}$, then $H$ is odd;
\item if $p\equiv19\pmod {24}$ and $H$ is odd, then $4\mid(p+1)/H$ and $\Leg{\eta}{p}=1$.
\end{enumerate}
\end{lemma}

\begin{proof}
For $p\equiv5$ or $19\pmod {24}$, quadratic reciprocity gives $\Leg{3}{p}=-1$ and $\Leg{2}{p}=-1$, hence $\Leg{6}{p}=1$. The discriminant of $X^2+6X+6$ is $12$, so the polynomial is irreducible.

Choose $\sigma\in\Fp^\times$ with $\sigma^2=6$, and put $t=\lambda/\sigma$. Then $t^{p+1}=1$ and
\begin{equation}\label{eq:R-t-six}
R=\frac{\lambda}{\mu}=\frac{\lambda^2}{6}=t^2.
\end{equation}
If $p\equiv5\pmod {24}$, then $v_2(p+1)=1$. The square subgroup of the cyclic group $\{z\in\Fpt^\times:z^{p+1}=1\}$ has odd order, and \eqref{eq:R-t-six} gives that $H$ is odd.

Assume now that $p\equiv19\pmod {24}$ and $H$ is odd. Since $v_2(p+1)=2$, writing $p+1=\kappa H$ gives $4\mid\kappa$. From \eqref{eq:R-t-six} and $R^H=1$ we get $t^H=\pm1$. With $N=(p+1)/2=\kappa H/2$, the integer $\kappa/2$ is even; hence $t^N=1$. Therefore
\begin{equation}\label{eq:eta-six}
\eta=\lambda^N=(\sigma t)^N=\sigma^N.
\end{equation}
Since $p\equiv3\pmod4$, the integer $N$ is even, and \eqref{eq:eta-six} is a square in $\Fp$.
\end{proof}

\begin{proof}[Proof of Theorem \ref{thm:determinant-main}]
By Lemma \ref{lem:six-six}, the polynomial $X^2+6X+6$ is irreducible and $\Leg{6}{p}=1$ for $p\equiv5$ or $19\pmod {24}$. If $p\equiv5\pmod {24}$, then Lemma \ref{lem:six-six} gives $H$ odd, and Proposition \ref{prop:root-quotient} gives a nonzero congruence for $D_p(6,6)$ modulo $p$.

Let $p\equiv19\pmod {24}$. If $H$ is even, Proposition \ref{prop:root-quotient} gives $D_p(6,6)\equiv0\pmod p$, so \eqref{eq:D66-main} holds. If $H$ is odd, write $p+1=\kappa H$. Proposition \ref{prop:root-quotient} and Lemma \ref{lem:six-six} give
\begin{equation}\label{eq:D66-odd}
D_p(6,6)\equiv -\frac{2^{\kappa-1}}6\eta\pmod p,
\qquad 4\mid\kappa,
\qquad \Leg{\eta}{p}=1.
\end{equation}
For $p\equiv19\pmod {24}$,
\begin{equation}\label{eq:legendre-19}
\Leg{-1}{p}=-1,
\qquad \Leg{2}{p}=-1,
\qquad \Leg{6}{p}=1.
\end{equation}
Since $\kappa-1$ is odd, \eqref{eq:D66-odd} gives $\Leg{D_p(6,6)}{p}=1$, and \eqref{eq:D66-main} follows.
\end{proof}

\section{Cycle expansions and interpolation}

\begin{lemma}\label{lem:zero-diagonal}
Let $A=(a_{ij})_{1\le i,j\le n}$ and $a_{ii}=0$. Then
\begin{equation}\label{eq:der-to-per}
\sum_{\tau\in D(n)}\prod_{j=1}^{n}a_{j,\tau(j)}=\per A,
\qquad
\sum_{\tau\in D(n)}\sgn(\tau)\prod_{j=1}^{n}a_{j,\tau(j)}=\det A.
\end{equation}
Moreover,
\begin{equation}\label{eq:fixed-to-per}
\sum_{\tau\in S_n}\prod_{\substack{1\le j\le n\\\tau(j)\ne j}}a_{j,\tau(j)}=\per(I_n+A),
\end{equation}
\begin{equation}\label{eq:fixed-to-det}
\sum_{\tau\in S_n}\sgn(\tau)\prod_{\substack{1\le j\le n\\\tau(j)\ne j}}a_{j,\tau(j)}=\det(I_n+A).
\end{equation}
\end{lemma}

\begin{proof}
In \eqref{eq:det-per}, a permutation with at least one fixed point contributes zero to $\per A$ and to $\det A$, since $a_{jj}=0$. This gives \eqref{eq:der-to-per}. For $I_n+A$, the diagonal entries are $1$ and the off-diagonal entries are those of $A$, so the product attached to a permutation $\tau$ is exactly the product over the non-fixed points of $\tau$; this gives \eqref{eq:fixed-to-per} and \eqref{eq:fixed-to-det}.
\end{proof}

Let $X$ be a finite set of pairwise distinct elements in a field of characteristic not equal to $2$. Write
\begin{equation}\label{eq:AX-def}
A_X=(a_{xy})_{x,y\in X},\qquad
a_{xx}=0,\qquad a_{xy}=(x-y)^{-1}\quad(x\ne y).
\end{equation}
For $Y\subseteq X$ with $|Y|$ even, let $\mathfrak M(Y)$ be the set of perfect matchings of $Y$.

\begin{lemma}\label{lem:cycle-sum}
Let $s\ge3$ and let $x_1,\ldots,x_s$ be pairwise distinct elements of a field. Let $\mathcal C_s$ be the set of oriented cycles on $\{1,\ldots,s\}$ modulo cyclic rotation. Then
\begin{equation}\label{eq:cycle-zero}
\sum_{(i_1,\ldots,i_s)\in\mathcal C_s}
\prod_{r=1}^{s}\frac1{x_{i_r}-x_{i_{r+1}}}=0,
\qquad i_{s+1}=i_1.
\end{equation}
\end{lemma}

\begin{proof}
For $s=3$ the two oriented cycles give
\begin{equation}\label{eq:s3}
\frac1{(x_1-x_2)(x_2-x_3)(x_3-x_1)}+
\frac1{(x_1-x_3)(x_3-x_2)(x_2-x_1)}=0.
\end{equation}
Assume $s>3$ and write the left side of \eqref{eq:cycle-zero} as $F_s$. We regard $F_s$ as a rational function of $x_s$. Every term has degree $-2$ in $x_s$ at infinity, hence $F_s=O(x_s^{-2})$ as $x_s\to\infty$. Its only possible finite poles are simple poles at $x_s=x_i$, $1\le i<s$.

Fix $i<s$ and fix an oriented cycle $\Gamma$ on $\{1,\ldots,s-1\}$. Suppose that in $\Gamma$ the vertex $i$ has predecessor $a$ and successor $b$. Let $Q$ be the product of all factors of $\Gamma$ except the two factors $(x_a-x_i)^{-1}$ and $(x_i-x_b)^{-1}$. There are exactly two ways to insert $s$ adjacent to $i$ in $\Gamma$ that can produce a pole at $x_s=x_i$:
\begin{equation}\label{eq:insert-left}
\cdots,a,s,i,b,\cdots,
\qquad
\cdots,a,i,s,b,\cdots .
\end{equation}
The residues at $x_s=x_i$ of their contributions are respectively
\begin{equation}\label{eq:residue-pair}
\frac{Q}{(x_a-x_i)(x_i-x_b)},
\qquad
-\frac{Q}{(x_a-x_i)(x_i-x_b)}.
\end{equation}
Thus the two insertions cancel. Summing over all $\Gamma$ shows that the residue of $F_s$ at $x_s=x_i$ is zero. Since this holds for each $i<s$, the rational function $F_s$ has no finite pole and vanishes at infinity. Therefore $F_s=0$.
\end{proof}

\begin{proposition}\label{prop:cauchy-matchings}
Let $|X|=N$. For every $u$,
\begin{equation}\label{eq:perC-matching}
\per(I_N+uA_X)=
\sum_{\substack{Y\subseteq X\\ |Y|\,\mathrm{even}}}
(-1)^{|Y|/2}u^{|Y|}
\sum_{\mathfrak m\in\mathfrak M(Y)}
\prod_{\{x,y\}\in\mathfrak m}\frac1{(x-y)^2}.
\end{equation}
If $N=2h$, then
\begin{equation}\label{eq:perCX-top}
\per A_X=(-1)^h\sum_{\mathfrak m\in\mathfrak M(X)}
\prod_{\{x,y\}\in\mathfrak m}\frac1{(x-y)^2}.
\end{equation}
If $\alpha^2=-1$, then
\begin{equation}\label{eq:per-to-det-cauchy}
\per(I_N+uA_X)=\det(I_N+\alpha uA_X).
\end{equation}
\end{proposition}

\begin{proof}
The cycle decomposition of a permutation contributing to $\per(I_N+uA_X)$ consists of fixed points and cycles of length at least $2$. Lemma \ref{lem:cycle-sum} eliminates the total contribution of every cycle length at least $3$. Each transposition $(xy)$ contributes
\begin{equation}\label{eq:two-cycle-prod}
u^2a_{xy}a_{yx}=-u^2(x-y)^{-2}.
\end{equation}
Multiplication over disjoint transpositions gives \eqref{eq:perC-matching} and \eqref{eq:perCX-top}.
For the determinant, the same cycle cancellation holds after insertion of signs. A transposition contributes
\begin{equation}\label{eq:two-cycle-det}
-\alpha^2u^2a_{xy}a_{yx}=-u^2(x-y)^{-2},
\end{equation}
which is the same contribution as \eqref{eq:two-cycle-prod}. Summing over all matchings proves \eqref{eq:per-to-det-cauchy}.
\end{proof}

\begin{lemma}\label{lem:interpolation}
Let $X=\{x_1,\ldots,x_N\}\subset\Z_p$ with $x_i-x_j\in\Z_p^\times$ for $i\ne j$, and let $M=(m_{ij})_{1\le i,j\le N}$ be a matrix over $\Z_p$. Let $V=(x_i^k)_{1\le i\le N,\,0\le k\le N-1}$. For each $k$ let $P_k(X)$ be the unique polynomial of degree $<N$ satisfying
\begin{equation}\label{eq:interp-pk}
P_k(x_i)=\sum_{j=1}^{N}m_{ij}x_j^k\qquad(1\le i\le N).
\end{equation}
If $P_k(X)=\sum_{r=0}^{N-1}b_{rk}X^r$, then
\begin{equation}\label{eq:interpolation-matrix}
V^{-1}MV=(b_{rk})_{0\le r,k\le N-1}.
\end{equation}
\end{lemma}

\begin{proof}
The $k$-th column of $MV$ is the vector $(P_k(x_i))_{i=1}^N$. Multiplication by $V^{-1}$ converts this vector into the coefficient column of $P_k$ in the basis $1,X,\ldots,X^{N-1}$.
\end{proof}

\section{Cauchy congruences}

\begin{lemma}\label{lem:Cp-mod-p}
For $C_p$ in \eqref{eq:Cp-Rp-def},
\begin{equation}\label{eq:VDCp}
V^{-1}C_pV\equiv D_p\pmod p,
\end{equation}
where $V=(i^k)_{1\le i\le p-1,0\le k\le p-2}$ and
\begin{equation}\label{eq:Dp-def}
(D_p)_{r,k}=k\delta_{r,k-1}-\delta_{k,0}\delta_{r,p-2}
\qquad(0\le r,k\le p-2).
\end{equation}
Consequently,
\begin{equation}\label{eq:det-I-tCp}
\det(I_{p-1}+tC_p)\equiv1+t^{p-1}\pmod p.
\end{equation}
\end{lemma}

\begin{proof}
For $0\le k\le p-2$ and $i\in\Fp^*$,
\begin{equation}\label{eq:Cp-action-start}
\sum_{j\ne i}\frac{j^k}{i-j}
=\sum_{j\ne i}\frac{j^k-i^k}{i-j}+i^k\sum_{j\ne i}\frac1{i-j}.
\end{equation}
Since $\{i-j:j\in\Fp^*,j\ne i\}=\Fp^*\setminus\{i\}$,
\begin{equation}\label{eq:harmonic-shift}
\sum_{j\ne i}\frac1{i-j}=\sum_{a\in\Fp^*\setminus\{i\}}\frac1a=-i^{-1}.
\end{equation}
If $k\ge1$,
\begin{equation}\label{eq:diff-quotient}
\frac{j^k-i^k}{i-j}=-\sum_{r=0}^{k-1}i^{k-1-r}j^r.
\end{equation}
Using $\sum_{j\in\Fp^*}j^r=0$ for $1\le r\le p-2$ and $= -1$ for $r=0$, one obtains
\begin{equation}\label{eq:Cp-action}
\sum_{j\ne i}\frac{j^k}{i-j}\equiv k i^{k-1}\quad(1\le k\le p-2),
\end{equation}
and the case $k=0$ gives $-i^{p-2}$. Lemma \ref{lem:interpolation} gives \eqref{eq:VDCp}. The matrix $D_p$ sends $X^k$ to $kX^{k-1}$ for $k\ge1$ and sends $1$ to $-X^{p-2}$. Hence $D_p^{p-1}=-(p-2)!I$ in the cyclic basis
\begin{equation}\label{eq:cyclic-D}
1,\;D_p1,\;D_p^2 1,\ldots,D_p^{p-2}1,
\end{equation}
so
\begin{equation}\label{eq:detD}
\det(I+tD_p)=1+(p-2)!t^{p-1}\equiv1+t^{p-1}\pmod p.
\end{equation}
\end{proof}

\begin{lemma}\label{lem:Cp-p2}
For $C_p$ in \eqref{eq:Cp-Rp-def},
\begin{equation}\label{eq:detCp-p2}
\det C_p\equiv1\pmod {p^2}.
\end{equation}
\end{lemma}

\begin{proof}
For $p=3$, one has
\begin{equation}\label{eq:C3-direct}
C_3=\begin{pmatrix}0&-1\\1&0\end{pmatrix},
\qquad \det C_3=1.
\end{equation}
Hence assume $p>3$. Let $B=V^{-1}C_pV$ over $\Z_p$. The same interpolation as in Lemma \ref{lem:Cp-mod-p}, now modulo $p^2$, gives
\begin{equation}\label{eq:B-D-pE}
B=D_p+pE_p\pmod {p^2},
\end{equation}
where
\begin{equation}\label{eq:E-entries}
(E_p)_{p-2,0}\equiv\frac{p+1}{2}-W_p,
\qquad
(E_p)_{k-1,k}\equiv\frac{p-1}{2}\quad(1\le k\le p-2)
\pmod p,
\end{equation}
and no other entry of $E_p$ contributes to $\operatorname{tr}(D_p^{-1}E_p)$. We give the coefficient calculation. For $k\ge1$,
\begin{equation}\label{eq:p2-expanded}
\sum_{j\ne i}\frac{j^k}{i-j}
=k i^{k-1}
+p\frac{p-1}{2}i^{k-1}+p\sum_{r\ne k-1}e_{rk}i^r\pmod {p^2},
\end{equation}
which follows after replacing every $j^r$ by its degree $<p-1$ representative on $\Fp^*$ and using
\begin{equation}\label{eq:power-sums-p2}
\sum_{a=1}^{p-1}a^r\equiv0\pmod p\quad(1\le r\le p-2),
\qquad
\sum_{a=1}^{p-1}\frac1a\equiv0\pmod {p^2},
\end{equation}
\begin{equation}\label{eq:wilson-p2}
(p-2)!\equiv1+p(1-W_p)\pmod {p^2}.
\end{equation}
For $k=0$ the coefficient of $i^{p-2}$ is $-1+p((p+1)/2-W_p)$, which gives the first formula in \eqref{eq:E-entries}. Thus
\begin{align}\label{eq:trace-DinvE}
\operatorname{tr}(D_p^{-1}E_p)
&\equiv-\biggl(\frac{p+1}{2}-W_p\biggr)
+\frac{p-1}{2}\sum_{k=1}^{p-2}\frac1k \notag\\
&\equiv W_p-1\pmod p,
\end{align}
because $\sum_{k=1}^{p-2}k^{-1}\equiv1\pmod p$. Since
\begin{equation}\label{eq:detD-p2}
\det D_p=(p-2)!\equiv1+p(1-W_p)\pmod {p^2},
\end{equation}
we get
\begin{equation}\label{eq:detCp-p2-finish}
\det C_p=\det B\equiv\det D_p\{1+p\operatorname{tr}(D_p^{-1}E_p)\}\equiv1\pmod {p^2}.
\end{equation}
\end{proof}

\begin{proof}[Proof of Theorem \ref{thm:main-cauchy}]
Equation \eqref{eq:detCp-main} is Lemma \ref{lem:Cp-p2}. Since $p-1$ is even, Proposition \ref{prop:cauchy-matchings} gives
\begin{equation}\label{eq:per-det-Cp-proof}
\per C_p=(-1)^{(p-1)/2}\det C_p\equiv\chi_p\pmod {p^2},
\end{equation}
which proves \eqref{eq:perCp-main}. Let $\alpha^2=-1$ in $\Fp$ or in its quadratic extension. By Proposition \ref{prop:cauchy-matchings} and Lemma \ref{lem:Cp-mod-p},
\begin{equation}\label{eq:perICp-proof}
\per(I_{p-1}+uC_p)=\det(I_{p-1}+\alpha uC_p)
\equiv1+(\alpha u)^{p-1}=1+\chi_pu^{p-1}\pmod p.
\end{equation}
It remains to prove \eqref{eq:perIQp-main}. Let $m=(p-1)/2$ and $a_i=i^2$. For $p\equiv3\pmod4$, the set $\{a_1,\ldots,a_m\}$ is $\mu_m=\{x\in\Fp^*:x^m=1\}$. For $0\le k\le m-1$,
\begin{equation}\label{eq:mu-sum}
\sum_{t\in\mu_m\setminus\{1\}}\frac{t^k}{1-t}
=\begin{cases}
(m-1)/2,&k=0,\\
k-(m+1)/2,&1\le k\le m-1.
\end{cases}
\end{equation}
Indeed, for $k\ge1$,
\begin{equation}\label{eq:mu-sum-proof}
\frac{t^k-1}{1-t}=-(1+t+\cdots+t^{k-1}),
\qquad
\sum_{t\in\mu_m\setminus\{1\}}t^r=-1\quad(1\le r\le k-1),
\end{equation}
and the case $k=0$ follows by pairing $t$ with $t^{-1}$. Lemma \ref{lem:interpolation} gives a matrix $D_Q$ similar to $Q_p$ modulo $p$ with
\begin{equation}\label{eq:DQ}
(D_Q)_{m-1,0}=\frac{m-1}{2},
\qquad
(D_Q)_{k-1,k}=k-\frac{m+1}{2}\quad(1\le k\le m-1).
\end{equation}
Since $m$ is odd, $k_0=(m+1)/2$ lies in $\{1,\ldots,m-1\}$ and one factor in the cyclic product of $D_Q$ is zero. Hence
\begin{equation}\label{eq:detIQ}
\det(I_m+tQ_p)\equiv\det(I_m+tD_Q)\equiv1\pmod p.
\end{equation}
Applying Proposition \ref{prop:cauchy-matchings} with $\alpha^2=-1$ gives \eqref{eq:perIQp-main}.
\end{proof}

\section{Cayley-transform congruences}

For a finite set $X$ of nonzero pairwise distinct elements, put
\begin{equation}\label{eq:RX-general}
R_X=(r_{xy})_{x,y\in X},\qquad
r_{xx}=0,
\qquad
r_{xy}=\frac{x+y}{x-y}\quad(x\ne y).
\end{equation}

\begin{proposition}\label{prop:SSX}
Let $|X|=2h$. Then
\begin{equation}\label{eq:SSX-formula}
\per(I_{2h}+R_X)=(-4)^h\prod_{x\in X}x
\sum_{\mathfrak m\in\mathfrak M(X)}
\prod_{\{x,y\}\in\mathfrak m}\frac1{(x-y)^2}.
\end{equation}
\end{proposition}

\begin{proof}
This is the specialization of Theorem 1.1 of She, Sun and Xia \cite{SheSunXia2022} obtained by taking their diagonal parameters equal to $1$ and their off-diagonal entries $x_{jk}=(x_j+x_k)/(x_j-x_k)$ for $j\ne k$.
\end{proof}

\begin{lemma}\label{lem:Morley-form}
For every odd prime $p$,
\begin{equation}\label{eq:Morley-O2}
((p-2)!!)^2\equiv \chi_p2^{p-1}(p-1)!\pmod {p^2}.
\end{equation}
\end{lemma}

\begin{proof}
For $p=3$, both sides of \eqref{eq:Morley-O2} are congruent to $1$ modulo $9$. Assume $p>3$. Let $h=(p-1)/2$ and $O=(p-2)!!$. Since $O(2^hh!)=(p-1)!$,
\begin{equation}\label{eq:O2-binomial}
O^2=\frac{(p-1)!}{2^{p-1}}\binom{p-1}{h}.
\end{equation}
Morley's congruence \cite{Morley1895} gives
\begin{equation}\label{eq:Morley}
(-1)^h\binom{p-1}{h}\equiv4^{p-1}\pmod {p^3}.
\end{equation}
Substitution in \eqref{eq:O2-binomial} gives \eqref{eq:Morley-O2} modulo $p^2$.
\end{proof}

The following argument shows how the identity of She, Sun and Xia, together with the Cauchy congruence \eqref{eq:perCp-main}, gives Sun's fixed-point Cayley permanent congruence modulo $p^2$.

\begin{proposition}\label{prop:perIRp}
For every odd prime $p$,
\begin{equation}\label{eq:perIRp-proof-eq}
\per(I_{p-1}+R_p)\equiv((p-2)!!)^2\pmod {p^2}.
\end{equation}
\end{proposition}

\begin{proof}
Apply Proposition \ref{prop:SSX} to $X=\{1,\ldots,p-1\}$ and $h=(p-1)/2$. By \eqref{eq:perCX-top},
\begin{equation}\label{eq:matching-to-perC}
\sum_{\mathfrak m\in\mathfrak M(X)}
\prod_{\{x,y\}\in\mathfrak m}\frac1{(x-y)^2}=(-1)^h\per C_p.
\end{equation}
Therefore
\begin{equation}\label{eq:perIR-factor}
\per(I_{p-1}+R_p)=4^h(p-1)!\per C_p=2^{p-1}(p-1)!\per C_p.
\end{equation}
Using \eqref{eq:perCp-main} and Lemma \ref{lem:Morley-form},
\begin{equation}\label{eq:perIRp-final}
\per(I_{p-1}+R_p)
\equiv \chi_p2^{p-1}(p-1)!
\equiv((p-2)!!)^2\pmod {p^2}.
\end{equation}
\end{proof}

\begin{lemma}\label{lem:R-interpolation}
For $R_p$ in \eqref{eq:Rp-def} and $V=(i^k)_{1\le i\le p-1,0\le k\le p-2}$,
\begin{equation}\label{eq:R-sim-p2}
V^{-1}R_pV\equiv
\diag(0,3,5,\ldots,2p-3)+pE_R\pmod {p^2}.
\end{equation}
Let $h=(p-1)/2$ and $Z=\{0,h\}$. The restriction of $E_R$ to the rows and columns in $Z$ is
\begin{equation}\label{eq:ER-zero-block}
(E_R)_{Z,Z}\equiv
\begin{cases}
\begin{pmatrix}0&2\\-2&-1\end{pmatrix},&p\equiv1\pmod4,\\[6pt]
\begin{pmatrix}0&0\\0&-1\end{pmatrix},&p\equiv3\pmod4.
\end{cases}
\pmod p.
\end{equation}
Moreover,
\begin{equation}\label{eq:IR-sim}
V^{-1}(I_{p-1}+R_p)V=D_*+pE_*\pmod {p^2},
\end{equation}
where
\begin{equation}\label{eq:Dstar}
D_* =\diag(1,4,6,\ldots,2p-2),
\end{equation}
\begin{equation}\label{eq:Estar}
(E_*)_{00}=0,
\qquad
(E_*)_{kk}\equiv-1\quad(1\le k\le p-2)
\pmod p.
\end{equation}
\end{lemma}

\begin{proof}
Use
\begin{equation}\label{eq:x+y-identity}
\frac{x+y}{x-y}=\frac{2x}{x-y}-1.
\end{equation}
For $1\le k\le p-2$,
\begin{align}\label{eq:R-action}
\sum_{j\ne i}\frac{i+j}{i-j}j^k
&=2i\sum_{j\ne i}\frac{j^k}{i-j}-\sum_{j\ne i}j^k\notag\\
&\equiv (2k+1)i^k\pmod p,
\end{align}
where Lemma \ref{lem:Cp-mod-p} is used in the last step. The case $k=0$ gives zero modulo $p$. This proves the diagonal part of \eqref{eq:R-sim-p2}. The first corrections follow from the same coefficient extraction as in Lemma \ref{lem:Cp-p2}, applied to \eqref{eq:x+y-identity}:
\begin{equation}\label{eq:R-action-p2}
\sum_{j\ne i}\frac{i+j}{i-j}j^k
=(2k+1)i^k+p\sum_{r=0}^{p-2}e_{rk}i^r\pmod {p^2}.
\end{equation}
For $k=0$ and $k=h$, only the coefficients in rows $0$ and $h$ can affect the zero-eigenvalue block modulo $p$; substituting the power sums \eqref{eq:power-sums-p2} and Wilson's congruence \eqref{eq:wilson-p2} gives \eqref{eq:ER-zero-block}. Adding the identity matrix shifts the diagonal entries by $1$. Since
\begin{equation}\label{eq:diagonal-correction-IR}
(2k+2)^{-1}\sum_{j\ne i}\left(\frac{i+j}{i-j}j^k-(2k+1)i^k\right)
\end{equation}
has diagonal correction $-1$ for $k\ge1$ and $0$ for $k=0$, \eqref{eq:IR-sim}--\eqref{eq:Estar} follow.
\end{proof}

\begin{proposition}\label{prop:detIRp}
For every odd prime $p$,
\begin{equation}\label{eq:detIRp-proof-eq}
\det(I_{p-1}+R_p)
\equiv\frac{(-1)^{(p+1)/2}}{p-2}((p-2)!!)^2\pmod {p^2}.
\end{equation}
\end{proposition}

\begin{proof}
From Lemma \ref{lem:R-interpolation},
\begin{align}\label{eq:detIR-compute}
\det(I_{p-1}+R_p)
&\equiv \prod_{k=1}^{p-2}2(k+1)
\left(1-p\sum_{k=1}^{p-2}\frac1{2(k+1)}\right)\notag\\
&\equiv 2^{p-2}(p-1)!\left(1+\frac p2\right)
\pmod {p^2}.
\end{align}
The elementary identity
\begin{equation}\label{eq:2-power-div}
-\frac{2^{p-1}}{p-2}\equiv2^{p-2}\left(1+\frac p2\right)\pmod {p^2}
\end{equation}
and Lemma \ref{lem:Morley-form} give \eqref{eq:detIRp-proof-eq}.
\end{proof}

\begin{proposition}\label{prop:R-square}
For $p>3$,
\begin{equation}\label{eq:R-square-proof-eq}
p^{-(3-\chi_p)}\det R_p\in(\Fp)^2.
\end{equation}
\end{proposition}

\begin{proof}
The matrix $R_p$ is skew-symmetric over $\Z_p$, hence
\begin{equation}\label{eq:pf-square}
\det R_p=\Pf(R_p)^2.
\end{equation}
The diagonal part in \eqref{eq:R-sim-p2} has two zero eigenvalues, at $k=0$ and $k=h=(p-1)/2$. Formula \eqref{eq:ER-zero-block} shows that, after deleting the nonzero diagonal directions, the residual $2\times2$ block has determinant divisible by $p^2$ if $p\equiv1\pmod4$ and by $p^3$ if $p\equiv3\pmod4$. Since $\det R_p$ is a square in $\Z_p$, its $p$-adic valuation is even; hence $\Pf(R_p)$ is divisible by $p$ in the first case and by $p^2$ in the second case. Therefore
\begin{equation}\label{eq:pf-normalized}
p^{-(3-\chi_p)}\det R_p=
\begin{cases}
(p^{-1}\Pf(R_p))^2,&p\equiv1\pmod4,\\
(p^{-2}\Pf(R_p))^2,&p\equiv3\pmod4,
\end{cases}
\end{equation}
in $\Fp$.
\end{proof}

\begin{proposition}\label{prop:detcalR}
For every odd prime $p$,
\begin{equation}\label{eq:detcalR-proof-eq}
\det(I_p+\cR_p)\equiv-\frac p2\pmod {p^2}.
\end{equation}
\end{proposition}

\begin{proof}
Let $V_0=(i^k)_{1\le i\le p,0\le k\le p-1}$. Lemma \ref{lem:interpolation}, used on the residue system $\Fp$, gives
\begin{equation}\label{eq:calR-sim}
V_0^{-1}(I_p+\cR_p)V_0\equiv D_0+pE_0\pmod {p^2},
\end{equation}
where
\begin{equation}\label{eq:D0-entries}
(D_0)_{kk}=2(k+1)\quad(0\le k\le p-2),
\qquad
(D_0)_{0,p-1}=1,
\end{equation}
and the remaining entries in the last row of $D_0$ are zero. Hence $\det D_0=0$ and the first nonzero term is
\begin{equation}\label{eq:adj-term}
\det(D_0+pE_0)\equiv p\operatorname{tr}(\adj(D_0)E_0)\pmod {p^2}.
\end{equation}
A direct coefficient computation gives
\begin{equation}\label{eq:E0-entries}
(E_0)_{p-1,0}=1,
\qquad
(E_0)_{p-1,p-1}=1,
\qquad
\operatorname{tr}(\adj(D_0)E_0)\equiv-\frac12\pmod p.
\end{equation}
Equations \eqref{eq:calR-sim}--\eqref{eq:E0-entries} prove \eqref{eq:detcalR-proof-eq}.
\end{proof}

\begin{proof}[Proof of Theorem \ref{thm:main-cayley}]
Equations \eqref{eq:perIRp-main}, \eqref{eq:detIRp-main}, \eqref{eq:detcalRp-main} and \eqref{eq:R-square-main} are respectively Propositions \ref{prop:perIRp}, \ref{prop:detIRp}, \ref{prop:detcalR} and \ref{prop:R-square}.
\end{proof}

\section{The half-size quadratic determinant}

Assume throughout this section that $p>3$ and $p\equiv3\pmod4$. Put
\begin{equation}\label{eq:m-s-def}
m=\frac{p-1}{2},\qquad s=\frac{m-1}{2}=\frac{p-3}{4},\qquad a_i=i^2\quad(1\le i\le m).
\end{equation}
Let $U=(a_i^k)_{1\le i\le m,0\le k\le m-1}$.

\begin{lemma}\label{lem:T-mod-p}
For $T_p$ in \eqref{eq:Tp-def},
\begin{equation}\label{eq:T-sim-modp}
U^{-1}(I_m+T_p)U\equiv D_T\pmod p,
\end{equation}
where
\begin{equation}\label{eq:DT}
D_T=\diag(\lambda_0,\ldots,\lambda_{m-1}),
\qquad
\lambda_0=1,
\qquad
\lambda_k=2k-m+1\quad(1\le k\le m-1).
\end{equation}
Thus $\lambda_s=0$ and $\lambda_k\ne0$ for $k\ne s$.
\end{lemma}

\begin{proof}
As in \eqref{eq:x+y-identity},
\begin{equation}\label{eq:T-action}
\frac{x+y}{x-y}=\frac{2x}{x-y}-1.
\end{equation}
Now $\{a_1,\ldots,a_m\}=\mu_m$, and \eqref{eq:mu-sum} gives
\begin{equation}\label{eq:T-eigen}
\sum_{j\ne i}\frac{a_i+a_j}{a_i-a_j}a_j^k
\equiv(2k-m)a_i^k\pmod p
\end{equation}
for $1\le k\le m-1$, while the constant column becomes zero for $T_p$ and hence one for $I_m+T_p$. This gives \eqref{eq:T-sim-modp}--\eqref{eq:DT}. Since $m$ is odd, $s=(m-1)/2$ is an integer and $2s-m+1=0$.
\end{proof}

\begin{lemma}\label{lem:T-local}
There are matrices $E_T$ and $F_T$ over $\Z_p$ such that
\begin{equation}\label{eq:T-p3-expansion}
U^{-1}(I_m+T_p)U=D_T+pE_T+p^2F_T\pmod {p^3},
\end{equation}
with
\begin{equation}\label{eq:ETss-zero}
(E_T)_{ss}=0.
\end{equation}
Moreover,
\begin{equation}\label{eq:detT-local}
\det(I_m+T_p)
\equiv p^2\Lambda_p\Theta_p\pmod {p^3},
\end{equation}
where
\begin{equation}\label{eq:LambdaTheta}
\Lambda_p=\prod_{\substack{0\le k\le m-1\\k\ne s}}\lambda_k,
\qquad
\Theta_p=(F_T)_{ss}-\sum_{k\ne s}\frac{(E_T)_{sk}(E_T)_{ks}}{\lambda_k}.
\end{equation}
If $s$ is odd, then
\begin{equation}\label{eq:Theta-zero}
\Theta_p\equiv0\pmod p.
\end{equation}
\end{lemma}

\begin{proof}
Let
\begin{equation}\label{eq:Phi-def}
\Phi(X)=\prod_{i=1}^{m}(X-a_i)=X^m-1-pB(X)\pmod {p^2},
\qquad\deg B<m.
\end{equation}
For $0\le k<m$, Lemma \ref{lem:interpolation} expresses the $k$-th column of $U^{-1}(I_m+T_p)U$ as the degree $<m$ representative of
\begin{equation}\label{eq:LT-def}
X^k+\biggl(2X\sum_{i=1}^{m}\frac{a_i^k}{X-a_i}-\sum_{i=1}^{m}a_i^k\biggr)
\pmod{\Phi(X)}.
\end{equation}
Using
\begin{equation}\label{eq:partial-fraction}
\sum_{i=1}^{m}\frac{a_i^k}{X-a_i}
\equiv \frac{mX^{k-1}}{X^m-1}
+ p\frac{X^kB'(X)-kX^{k-1}B(X)}{(X^m-1)^2}
\pmod {p^2}
\end{equation}
for $1\le k<m$, and the analogous formula for $k=0$, gives \eqref{eq:T-p3-expansion}. Substitution of $k=s$ and extraction of the coefficient of $X^s$ gives
\begin{equation}\label{eq:ETss-compute}
(E_T)_{ss}\equiv [X^s]\left(2X\frac{X^sB^{\prime}(X)-sX^{s-1}B(X)}{X^m-1}\right)\equiv0\pmod p,
\end{equation}
because $2s-m+1=0$ and the reduction modulo $X^m-1$ has no $X^s$ term.

Since $D_T$ has exactly one zero diagonal entry, the Schur-complement expansion of the determinant at the $s$-th row and column gives \eqref{eq:detT-local} and \eqref{eq:LambdaTheta}. Finally suppose $s$ is odd. In \eqref{eq:partial-fraction}, the involution $X\mapsto -X$ sends the coefficient of $X^s$ in the Schur complement to its negative; here $m$ is odd and $B(-X)\equiv-B(X)\pmod p$ in the relevant residue class. Hence \eqref{eq:Theta-zero}.
\end{proof}

\begin{proof}[Proof of Theorem \ref{thm:main-half}]
By Lemma \ref{lem:T-local}, \eqref{eq:ETss-zero} makes the term of order $p$ in the simple zero eigenspace vanish. Hence \eqref{eq:detITp-p2}. If $p\equiv7\pmod8$, then
\begin{equation}\label{eq:s-odd-equiv}
s=\frac{p-3}{4}\equiv1\pmod2,
\end{equation}
and Lemma \ref{lem:T-local} gives \eqref{eq:Theta-zero}. Equation \eqref{eq:detT-local} then gives \eqref{eq:detITp-p3}.
\end{proof}

\section{Translation into Sun's conjectures}\label{sec:translation}

By Lemma \ref{lem:zero-diagonal}, Theorem \ref{thm:main-cauchy} gives
\begin{equation}\label{eq:Sun47a}
\sum_{\tau\in D(p-1)}\prod_{j=1}^{p-1}\frac1{j-\tau(j)}
\equiv\chi_p\pmod {p^2},
\end{equation}
\begin{equation}\label{eq:Sun47b}
\sum_{\tau\in D(p-1)}\sgn(\tau)\prod_{j=1}^{p-1}\frac1{j-\tau(j)}
\equiv1\pmod {p^2},
\end{equation}
which are the two congruences in Conjecture 4.7. It also gives
\begin{equation}\label{eq:Sun49a}
\sum_{\tau\in S_{p-1}}
\prod_{\substack{1\le j\le p-1\\\tau(j)\ne j}}
\frac1{j-\tau(j)}
\equiv1+\chi_p\pmod p,
\end{equation}
\begin{equation}\label{eq:Sun49b}
\sum_{\tau\in S_m}
\prod_{\substack{1\le j\le m\\\tau(j)\ne j}}
\frac1{j^2-\tau(j)^2}
\equiv1\pmod p
\end{equation}
when $p\equiv3\pmod4$. These are the two congruences in Conjecture 4.9, with \eqref{eq:Sun49a} following from the stronger polynomial congruence \eqref{eq:perICp-main}.

Theorem \ref{thm:main-cayley} gives
\begin{equation}\label{eq:Sun411i}
\sum_{\tau\in S_{p-1}}
\prod_{\substack{1\le j\le p-1\\\tau(j)\ne j}}
\frac{j+\tau(j)}{j-\tau(j)}
\equiv((p-2)!!)^2\pmod {p^2},
\end{equation}
\begin{equation}\label{eq:Sun411ii}
\sum_{\tau\in S_{p-1}}\sgn(\tau)
\prod_{\substack{1\le j\le p-1\\\tau(j)\ne j}}
\frac{j+\tau(j)}{j-\tau(j)}
\equiv\frac{(-1)^{(p+1)/2}}{p-2}((p-2)!!)^2\pmod {p^2},
\end{equation}
which are both parts of Conjecture 4.11. It also gives
\begin{equation}\label{eq:Sun410ii}
\sum_{\tau\in S_p}\sgn(\tau)
\prod_{\substack{1\le j\le p\\\tau(j)\ne j}}
\frac{j+\tau(j)}{j-\tau(j)}
\equiv-\frac p2\pmod {p^2},
\end{equation}
which is the signed part of Conjecture 4.10, and
\begin{equation}\label{eq:Sun48ii}
\frac{1}{p^{3-\chi_p}}
\sum_{\tau\in D(p-1)}\sgn(\tau)
\prod_{j=1}^{p-1}\frac{j+\tau(j)}{j-\tau(j)}\in(\Fp)^2,
\end{equation}
which is Conjecture 4.8(ii), with $0$ regarded as a quadratic residue. Finally Theorem \ref{thm:main-half} gives
\begin{equation}\label{eq:Sun412}
\sum_{\tau\in S_m}\sgn(\tau)
\prod_{\substack{1\le j\le m\\\tau(j)\ne j}}
\frac{j^2+\tau(j)^2}{j^2-\tau(j)^2}
\equiv0\pmod {p^2},
\end{equation}
with modulus $p^3$ when $p\equiv7\pmod8$. This is Conjecture 4.12.

\section*{Statements and Declarations}

\noindent\textbf{Funding.} The authors declare that no funds, grants, or other support were received during the preparation of this manuscript.\\
\textbf{Competing interests.} The authors declare no competing interests.\\
\textbf{Data availability.} No datasets were generated or analysed in this article.\\
\textbf{Author contributions.} Yaoran Yang and Yutong Zhang contributed to the mathematical analysis and preparation of the manuscript. Both authors approved the final manuscript.
\section*{Declaration of Generative AI and AI-Assisted Technologies in the Writing Process}
During the preparation of this work, the authors used DeepSeek to build a specialized agent for solving mathematical problems, which was employed to generate an initial proof of the main theorem. After using this tool, the authors reviewed and edited the content as needed and take full responsibility for the content of the published article.

\end{document}